\documentclass[final,leqno,onefignum,onetabnum]{siamltex1213}

\def\dil{\Lam}

\def\capset{\Gamma}
\def\dualset{\capset^\ast}

\def\beaN{\setlength{\arraycolsep}{0.0em}\begin{eqnarray*}}
\def\eeaN{\end{eqnarray*}\setlength{\arraycolsep}{5pt}}
\def\bea{\setlength{\arraycolsep}{0.0em}\begin{eqnarray}}
\def\eea{\end{eqnarray}\setlength{\arraycolsep}{5pt}}

\def\be{\begin{equation}}
\def\ee{\end{equation}}

\def\dm{n}

\def\Rd{\RR^\dm}
\def\Zd{\ZZ^\dm}
\def\Cd{\CC^\dm}
\def\Td{\TT^\dm}
\def\CC{\mathbb{C}}
\def\NN{\mathbb{N}}
\def\RR{\mathbb{R}}
\def\ZZ{\mathbb{Z}}
\def\TT{\mathbb{T}}

\def\norm#1{\|#1\|}

\def\disp{\displaystyle}

\def\bks{\backslash}

\def\alp{\alpha}                

\def\gam{\gamma}                \def\Gam{\Gamma}
                
\def\eps{\varepsilon}

\def\tet{\theta}

               \def\Lam{\Lambda}

\def\ome{\omega}                \def\Ome{\Omega}

\def\hatf{{\widehat f}}

\def\hatphi{{\widehat\phi}}

\def\hat#1{{\widehat {#1}}}

\usepackage{color}
\usepackage{cite}
\usepackage{graphicx}
\usepackage{amssymb}
\usepackage{enumerate}
\usepackage{bm,bbm}
\usepackage{multirow}

\author{Youngmi Hur\footnotemark[2]
\and Kasso A.~Okoudjou\footnotemark[3]}

\title{Scaling Laplacian Pyramids}

\begin{document}

\maketitle
\slugger{mms}{xxxx}{xx}{x}{x--x}

\renewcommand{\thefootnote}{\fnsymbol{footnote}}
\footnotetext[2]{Department of Mathematics, Yonsei University, Seoul 120-749, Korea (yhur@yonsei.ac.kr). Research supported in part by Yonsei New Faculty Research Seed Money Grant.} 
\footnotetext[3]{Department of Mathematics, University of Maryland, College Park, MD 20742, USA (kasso@math.umd.edu).  Research supported in part by a RASA from the Graduate School of UMCP and by a grant from the Simons Foundation ($\# 319197$ to Kasso Okoudjou).}

\begin{abstract}
Laplacian pyramid based Laurent polynomial (LP$^2$) matrices  are generated by Laurent polynomial column vectors and  have long been studied in connection with Laplacian pyramidal algorithms in Signal Processing. In this paper,  we investigate when such matrices are scalable, that is when right multiplication by  Laurent polynomial diagonal matrices results in paraunitary matrices. The notion of scalability has recently been  introduced in the context of finite frame theory and can be considered as a preconditioning method for frames. This paper significantly extends the current research on scalable frames to the setting of polyphase representations of filter banks. Furthermore, as applications of our main results we propose new construction methods for tight wavelet filter banks and tight wavelet frames.
\end{abstract}



\pagestyle{myheadings}
\thispagestyle{plain}
\markboth{Youngmi Hur and Kasso A.~Okoudjou}{Scaling Laplacian Pyramids}

\section{Introduction}
\label{S:Intro}
Let $\Td$ be the set of all $z:=[z_1,\dots,z_\dm]^T\in\Cd$ with $|z_i|=1$, for all $i=1,\dots,n$. Here and below, $T$ is used to represent the matrix transpose. $\mathcal{M}_{q,p}(z)$ will denote the set of all $q\times p$ matrices whose entries are Laurent polynomials in $z\in \Td$ with real coefficients, and $\mathcal{M}_{q}(z):=\mathcal{M}_{q,1}(z)$ will denote the set of all column vectors of length $q$. In the sequel, unless specified otherwise, we assume that all the relations (such as identities, inequalities) among Laurent polynomial matrices in $\mathcal{M}_{q,p}(z)$  hold true for all $z\in\Td$.  

We are mainly interested in a family of Laurent polynomial matrices arising from the study of Laplacian pyramidal algorithms \cite{BA} using the polyphase representation \cite{V}. 
Given an integer  $q\ge 2$, consider a nonzero column vector with Laurent polynomial entries $H_0(z), H_1(z),\dots, H_{q-1}(z)$, denoted by 
$${\tt H}(z):=[H_0(z),H_1(z),\dots,H_{q-1}(z)]^T\in \mathcal{M}_{q}(z).$$ 
Note that $H_j(z)$, for $j=0,\dots, q-1$, is used to denote the $(j+1)$-th entry of the column vector ${\tt H}(z)$.
 We define
$$
\Phi_{\tt H}(z):=\left[\begin{array}{cc} 
{\tt H}(z)& {\tt I}- {\tt H}(z){\tt H}^\ast(z)\\
\end{array}\right]\in \mathcal{M}_{q\times (q+1)}(z),
$$
where ${\tt I}$ is the identity matrix and ${\tt H}^\ast(z)$ is the conjugate transpose of ${\tt H}(z)$, which is given as
$${\tt H}^\ast(z):=\overline{{\tt H}(z)}^T=[\overline{H_0(z)},\overline{H_1(z)},\dots,\overline{H_{q-1}(z)}]=[H_0(z^{-1}),H_1(z^{-1}),\dots,H_{q-1}(z^{-1})].$$
Here, $z^{-1}:=[z_1^{-1}, \dots, z_n^{-1}]^T \in \Td$, for $z=[z_1, \dots, z_n]^T\in \Td$. 
It is readily seen that 
\begin{equation}\label{lp2}
\Phi_{\tt H}(z)\left[\begin{array}{c} 
{\tt H}^\ast(z)\\ 
{\tt I}\\
\end{array}\right]={\tt I},\quad \forall z\in\Td.
\end{equation} Consequently, ${\rm rank\,}\Phi_{\tt H}(z)=q$ for all $z\in\Td$.  
From now on we shall refer to the matrix $\Phi_{\tt H}(z)$  as the {\it LP$^2$ matrix (of order $q$) associated with ${\tt H}(z)$}. 
The LP$^2$ matrices have been studied in connection with various wavelet constructions \cite{DV1,HR1,Hur,HPZ}.

The LP$^2$ matrix $\Phi_{\tt H}(z)$ is said to be {\footnote[5]{The notion of paraunitary is typically considered for square matrices, and we extend the notion in this paper to more general rectangular matrices.}{\it paraunitary,}} if  
\begin{equation}\label{paraunitary}
\Phi_{\tt H}(z)\Phi_{\tt H}^\ast(z)={\tt I}.
\end{equation}  
The class of paraunitary LP$^2$ matrices is fundamentally related to the theory of tight filter banks, \cite{DV1,HR1}. Indeed, recall that 
from any pair of matrices ${\tt A}(z)\in \mathcal{M}_{q\times p}(z)$, ${\tt B}(z)\in \mathcal{M}_{p\times q}(z)$ such that ${\tt A}(z){\tt B}(z)={\tt I}$, a {\it filter bank} satisfying the perfect reconstruction property can be constructed (see, e.g., \cite{V}). Moreover, when ${\tt A}(z)$ is paraunitary, i.e. ${\tt A}(z){\tt A}^\ast(z)={\tt I}$, the pair $({\tt A}(z), {\tt A}^\ast(z))$ gives rise to a {\it tight filter bank}. The interest in choosing  ${\tt A}(z)$ specifically as an LP$^2$ matrix $\Phi_{\tt H}(z)$ stems from the fact that all the filters in $\Phi_{\tt H}(z)$ are generated from a single filter associated with the vector ${\tt H}(z)$ \cite{Hur,HPZ}.

Furthermore, it is easy to see that the design of tight filter bank from a paraunitary LP$^2$ matrix $\Phi_{\tt H}(z)$ is equivalent to the existence of a column matrix ${\tt H}(z)$ such that   ${\tt H}^\ast(z){\tt H}(z)=1$, or equivalently, $\sum_{k=0}^{q-1}|H_k(z)|^2=1\, \textrm{for\, all}\, z \in \Td.$ In particular, unless the norm of the column vector ${\tt H}(z)$ is identically constant, the associated LP$^2$ matrix cannot be paraunitary. This is the case, for example, if we choose ${\tt H}(z)=[1,(1+z^{-1})/2]^T/\sqrt{2}$ (cf. Example 1 in Section~\ref{subS:example} with $k=1$). It is then natural to  ask whether  a column vector ${\tt H}(z)$ such that   ${\tt H}^\ast(z){\tt H}(z)\neq 1$ can be modified into a new column vector ${\tt \widetilde{H}}(z)$ for which ${\tt \widetilde{H}}^\ast(z){\tt \widetilde{H}}(z)= 1$ leading to a paraunitary LP$^2$ matrix $\Phi_{\tt  \widetilde{H}}(z)$. This is a special case of a more general question that asks  whether one can find matrices $M(z)$ whose entries are Laurent polynomials  such that 
$\Phi_{\tt H}(z)M(z)$ is paraunitary, i.e. $$[\Phi_{\tt H}(z)M(z)][M^\ast(z)\Phi_{\tt H}^\ast(z)]={\tt I}.$$

In the scalar case and from a numerical linear algebra point of view these types of questions have been extensively studied
in the framework of matrix preconditioning \cite{bb, Kchen}.   In the context of finite frames \cite{CK}  a special case of this question was considered under the term of \emph{scalable frames} which were introduced in \cite{KOPT}. In this setting one seeks nonnegative (scalar-valued) matrices $D$ that would make a frame with synthesis (real) matrix $\Phi$, a tight frame, i.e., one which satisfies $$\Phi D^2 \Phi^T=I.$$ More investigations on scalable frames appeared in \cite{cc12, cklmns12, kop}.

The goal of this paper is twofold:\newline
\noindent $\bullet$ On the one hand (see Section~\ref{S:LPmain}) we extend the concept of scalability to matrices with Laurent polynomials. We say that the matrix $\Phi_{\tt H}(z)$ is {\it scalable} if there exists a Laurent polynomial diagonal matrix $M(z)$ such that $\Phi_{\tt H}(z)M(z)$ is paraunitary. 
 To this end we first investigate the more general question of the existence of  a Laurent polynomial diagonal matrix $B(z)$ such that $\Phi_{\tt H}(z)B(z)\Phi_{\tt H}^\ast(z)={\tt I}$, and determine when such diagonal matrix $B(z)$ can be written as $B(z)=M(z) M^{\ast}(z)$ for some Laurent polynomial diagonal matrix $M(z)$.  We then describe when such a diagonal matrix $M(z)$ can be written as $M(z)={\rm diag }([m(z), 1, \dots, 1])$ where $m(z)$ is a Laurent polynomial. In the univariate case ($n=1$) we completely settle the problem  by relying on the Fej\'er-Riesz factorization Lemma \cite{Fejer,Riesz}. \newline
\noindent $\bullet$ On the other hand we use these scalability results to develop a new methodology for constructing tight wavelet filter banks and tight wavelet frames in Section~\ref{S:wavelet}.  We recall that a filter bank is typically referred to as a {\it wavelet filter bank} if each of its analysis and synthesis banks has exactly one lowpass filter and the rest are all highpass filters. We will review the fundamentals of filter banks in general, and wavelet filter banks in particular, in Section~\ref{S:wavelet}. One of our main results in this section concerns the  transformation of univariate  non-tight, wavelet frames into tight wavelet ones in such a way that the resulting refinable functions preserve most of the properties of the original refinable functions.  In a forthcoming work we hope to use multi-dimensional versions of the Fej\'er-Riesz factorization Lemma (\cite{GeWo, GL}) to extend our construction to the multi-dimensional cases.

\section{Scaling LP$^2$ matrices}
\label{S:LPmain}
In this section, we present our main results about LP$^2$ matrices, the first of which says that, for any LP$^2$ matrix $\Phi_{\tt H}(z)\in \mathcal{M}_{q\times (q+1)}(z)$, there exists a diagonal matrix $B(z)\in\mathcal{M}_{q+1,q+1}(z)$ such that 
\begin{equation}
\label{eqn:paraunitary}
\Phi_{\tt H}(z)B(z)\Phi_{\tt H}^\ast(z)={\tt I}.
\end{equation}
In fact, we  give an explicit formula for the matrix $B(z)$. But, 
before we state the result, we need the following set up, and we refer to \cite{kop} for details. Given an integer $q\geq 2$, let $d:=(q-1)(q+2)/2$ and define a function $F: \CC^q \to \CC^d$ by

$$F(x):=\left[\begin{array}{c} 
F_0(x)\\ 
F_1(x)\\
\vdots\\
F_{q-1}(x)
\end{array}\right]\in\CC^d,\quad 
\hbox{ for }
x=\left[\begin{array}{c} 
x_0\\ 
x_1\\
\vdots\\
x_{q-1}
\end{array}\right]\in\CC^q,$$
where $F_0(x)\in\CC^{q-1}$ and $F_k(x)\in\CC^{q-k}$, for $k=1,\dots,q-1$, are defined as
$$
F_0(x)=\left[\begin{array}{c} 
|x_0|^2-|x_1|^2\\ 
|x_0|^2-|x_2|^2\\
\vdots\\
|x_0|^2-|x_{q-1}|^2
\end{array}\right],\quad
F_k(x)=\left[\begin{array}{c} 
x_{k-1}\overline{x_k}\\ 
x_{k-1}\overline{x_{k+1}}\\
\vdots\\
x_{k-1}\overline{x_{q-1}}
\end{array}\right].
$$

Let ${\tt H}(z)=[H_0(z),H_1(z),\dots,H_{q-1}(z)]^T$. 
Following  \cite{kop}, we see that a diagonal Laurent polynomial matrix $B(z)$ with diagonals  $b_0(z), b_1(z), \dots, b_q(z)$   solves  
\begin{equation}
\label{eq:KOPT}
\Phi_{\tt H}(z) {\rm diag }([b_0(z), b_1(z), \dots, b_q(z)])\Phi_{\tt H}^\ast(z)={\tt I}
\end{equation} 
(or equivalently,~(\ref{eqn:paraunitary})) 
if and only if they satisfy 
\begin{eqnarray}
  &{}&|H_0(z)|^2b_0(z)+(1-|H_0(z)|^2)^2b_1(z)\label{eq:inhomo}\\
  &{}&\quad\quad +\,|H_0(z)|^2|H_1(z)|^2b_2(z)+\dots+|H_0(z)|^2|H_{q-1}(z)|^2b_q(z)=1\nonumber
\end{eqnarray}
and
\begin{equation}
\label{eq:FPhi}
F(\Phi_{\tt H}(z))\left[\begin{array}{c} 
b_0(z)\\ 
b_1(z)\\
\vdots\\
b_q(z)
\end{array}\right]=0
\end{equation}
where $F(\Phi_{\tt H}(z))\in\mathcal{M}_{d,q+1}(z)$ is obtained by applying the function $F$ to the column vectors of $\Phi_{\tt H}(z)$. 

 We can now state the main result of this section.
\begin{theorem}
\label{thm:N(z)}
Let $\Phi_{\tt H}(z)$ be an LP$^2$ matrix associated with ${\tt H}(z)\in \mathcal{M}_{q}(z)$. Then we have $$\Phi_{\tt H}(z) {\rm diag }([2-{\tt H}^\ast(z){\tt H}(z), 1, \dots, 1])\Phi_{\tt H}^\ast(z)={\tt I}.\qquad\endproof$$
\end{theorem}

In order to prove Theorem~\ref{thm:N(z)}, we  make some simple but key observations that allow us to find each entry of  the matrix $F(\Phi_{\tt H}(z))$ from the definitions of the function $F$ and the LP$^2$ matrix $\Phi_{\tt H}(z)$. 

\begin{lemma}
\label{lemma:F_0}
For $i=1,\dots,q-1$ and $j=1,\dots,q+1$, the $(i,j)$ entry of $F_0(\Phi_{\tt H}(z))\in \mathcal{M}_{q-1,q+1}(z)$ is
$$
\cases{|H_0(z)|^2-|H_i(z)|^2,& if $j=1$,\cr
(1-|H_0(z)|^2)^2-|H_i(z)|^2|H_0(z)|^2,& if $j=2$, \cr
|H_0(z)|^2|H_i(z)|^2-(1-|H_i(z)|^2)^2,& if $j=i+2$,\cr
(|H_0(z)|^2-|H_i(z)|^2)|H_{j-2}(z)|^2,& if $j\ge 3$ and $j\ne i+2$.\qquad\endproof
}
$$
\end{lemma}

\begin{proof}
The LP$^2$ matrix $\Phi_{\tt H}(z)\in \mathcal{M}_{q\times (q+1)}(z)$ can be written explicitly as
$$
\left[\begin{array}{ccccc} 
H_0(z) & 1-|H_0(z)|^2 & -H_0(z)\overline{H_1(z)} & \cdots & -H_0(z)\overline{H_{q-1}(z)}\\ 
H_1(z) &  -H_1(z)\overline{H_0(z)} & 1-|H_1(z)|^2 & \cdots &-H_1(z)\overline{H_{q-1}(z)} \\ 
\vdots &   &  & \ddots & \\
H_{q-1}(z) & -H_{q-1}(z)\overline{H_0(z)} & -H_{q-1}(z)\overline{H_1(z)} & \cdots & 1-|H_{q-1}(z)|^2\\
\end{array}\right].
$$
By definition, the first column of $F_0(\Phi_{\tt H}(z))$ is    
$$F_0\left(\left[\begin{array}{c} 
H_0(z) \\
H_1(z) \\
\vdots \\
H_{q-1}(z) 
\end{array}\right]\right)=\left[\begin{array}{c} 
|H_0(z)|^2- |H_1(z)|^2\\
|H_0(z)|^2- |H_2(z)|^2\\
\vdots \\
|H_0(z)|^2- |H_{q-1}(z)|^2\\
\end{array}\right],$$
and its second column is 
$$F_0\left(\left[\begin{array}{c} 
1-|H_0(z)|^2 \\ 
-H_1(z)\overline{H_0(z)} \\ 
\vdots \\
-H_{q-1}(z)\overline{H_0(z)} \\
\end{array}\right]\right)=\left[\begin{array}{c} 
(1-|H_0(z)|^2)^2-|H_1(z)|^2|H_0(z)|^2\\
(1-|H_0(z)|^2)^2-|H_2(z)|^2|H_0(z)|^2\\
\vdots \\
(1-|H_0(z)|^2)^2-|H_{q-1}(z)|^2|H_0(z)|^2\\
\end{array}\right].$$
Thus, for $i=1,\dots,q-1$, the $(i,1)$ entry of $F_0(\Phi_{\tt H}(z))$ is $|H_0(z)|^2-|H_i(z)|^2$, and the $(i,2)$ entry of $F_0(\Phi_{\tt H}(z))$ is $(1-|H_0(z)|^2)^2-|H_i(z)|^2|H_0(z)|^2$, as desired.

The third through the last column of $F_0(\Phi_{\tt H}(z))$ behave essentially the same, and the third column is given as 
$$F_0\left(\left[\begin{array}{c} 
-H_0(z)\overline{H_1(z)}\\ 
1-|H_1(z)|^2\\ 
\vdots \\
 -H_{q-1}(z)\overline{H_1(z)} \\
\end{array}\right]\right)=\left[\begin{array}{c} 
|H_0(z)|^2|H_1(z)|^2-(1-|H_1(z)|^2)^2\\
|H_0(z)|^2|H_1(z)|^2-|H_2(z)|^2|H_1(z)|^2\\
\vdots \\
|H_0(z)|^2|H_1(z)|^2-|H_{q-1}(z)|^2|H_1(z)|^2\\
\end{array}\right],$$
and the last column is given as
$$F_0\left(\left[\begin{array}{c} 
-H_0(z)\overline{H_{q-1}(z)}\\ 
-H_1(z)\overline{H_{q-1}(z)} \\
\vdots \\
1-|H_{q-1}(z)|^2\\ 
\end{array}\right]\right)=\left[\begin{array}{c} 
|H_0(z)|^2|H_{q-1}(z)|^2-|H_1(z)|^2|H_{q-1}(z)|^2\\
|H_0(z)|^2|H_{q-1}(z)|^2-|H_2(z)|^2|H_{q-1}(z)|^2\\
\vdots \\
|H_0(z)|^2|H_{q-1}(z)|^2-(1-|H_{q-1}(z)|^2)^2\\
\end{array}\right].$$
Hence, for $i=1,\dots,q-1$ and $j=3,\dots,q+1$, the $(i,j)$ entry of $F_0(\Phi_{\tt H}(z))$ is
$$\cases{  |H_0(z)|^2|H_i(z)|^2-(1-|H_i(z)|^2)^2,& if $j=i+2$,\cr
               (|H_0(z)|^2-|H_i(z)|^2)|H_{j-2}(z)|^2,& if $j\ne i+2$,\cr
            }
$$            
as desired.
\end{proof}

\begin{lemma}
\label{lemma:F_k}
Let $k=1,\dots,q-1$ be fixed. For $i=1,\dots,q-k$ and $j=1,\dots,q+1$, the $(i,j)$ entry of $F_k(\Phi_{\tt H}(z))\in \mathcal{M}_{q-k,q+1}(z)$ is 
$$
\cases{H_{k-1}(z)\overline{H_{i+k-1}(z)},& if $j=1$,\cr
-(1-|H_{j-2}(z)|^2)H_{k-1}(z)\overline{H_{i+k-1}(z)},& if $j=k+1$ or $j=i+k+1$, \cr
H_{k-1}(z)\overline{H_{i+k-1}(z)}|H_{j-2}(z)|^2,& if $j\ge 2$, $j\ne k+1$ and $j\ne i+k+1$.\qquad\endproof
}
$$
\end{lemma}

\begin{proof}
By definition, the first column of $F_k(\Phi_{\tt H}(z))$ is
$$F_k\left(\left[\begin{array}{c} 
H_0(z) \\
H_1(z) \\
\vdots \\
H_{q-1}(z) 
\end{array}\right]\right)=\left[\begin{array}{c} 
H_{k-1}(z)\overline{H_k(z)}\\
H_{k-1}(z)\overline{H_{k+1}(z)}\\
\vdots \\
H_{k-1}(z)\overline{H_{q-1}(z)}\\
\end{array}\right].$$
Hence, for $i=1,\cdots,q-k$, the $(i,1)$ entry of $F_k(\Phi_{\tt H}(z))$ is $H_{k-1}(z)\overline{H_{i+k-1}(z)}$, as desired.

For $i=1,\cdots,q-k$, the second to the last entry of the $i$-th row of $F_k(\Phi_{\tt H}(z))$ is the $i$-th row of $F_k({\tt I}-{\tt H}^\ast(z){\tt H}(z))$, which can be obtained by multiplying the $k$-th row of ${\tt I}-{\tt H}^\ast(z){\tt H}(z)$ and the complex conjugate of $(i+k)$-th row of ${\tt I}-{\tt H}^\ast(z){\tt H}(z)$ entry-wise. Since, for $j=2,\dots,q+1$, the $(j-1)$-th entry of the $k$-th row of ${\tt I}-{\tt H}^\ast(z){\tt H}(z)$ is 
$$\cases{  -H_{k-1}(z)\overline{H_{j-2}(z)},& if $j\ne k+1$,\cr
           1-|H_{j-2}(z)|^2,& if $j= k+1$,\cr
            }
$$           
and the $(j-1)$-th entry of the complex conjugate of $(i+k)$-th row of ${\tt I}-{\tt H}^\ast(z){\tt H}(z)$ is
$$\cases{  -\overline{H_{i+k-1}(z)}H_{j-2}(z),& if $j\ne i+k+1$,\cr
           1-|H_{j-2}(z)|^2,& if $j= i+k+1$,\cr
            }
$$           
the $(i,j-1)$ entry of $F_k({\tt I}-{\tt H}^\ast(z){\tt H}(z))$, or equivalently, the $(i,j)$ entry of $F_k(\Phi_{\tt H}(z))$, is
$$\cases{  H_{k-1}(z)\overline{H_{i+k-1}(z)}|H_{j-2}(z)|^2,& if $j\ne k+1$ and $j\ne i+k+1$,\cr
           -(1-|H_{j-2}(z)|^2)H_{k-1}(z)\overline{H_{i+k-1}(z)},& if $j= k+1$ or $j=i+k+1$,\cr
            }
$$           
as desired.
\end{proof}

Now we are ready to present our proof of Theorem~\ref{thm:N(z)}.

{\em Proof of Theorem~\ref{thm:N(z)}.}
From Lemma~\ref{lemma:F_0}, we see that, for $i=1,\dots,q-1$, the $i$-th row of $F_0(\Phi_{\tt H}(z))\in \mathcal{M}_{q-1,q+1}(z)$ has the following form:
\begin{enumerate}[(i)]
\item the first entry is $|H_0(z)|^2-|H_i(z)|^2$,
\item the second entry is $(1-|H_0(z)|^2)^2-|H_i(z)|^2|H_0(z)|^2$,
\item the $j$-th entry, for $j=3,\dots,q+1$, is 
$$\cases{  |H_0(z)|^2|H_i(z)|^2-(1-|H_i(z)|^2)^2,& if $j=i+2$,\cr
               (|H_0(z)|^2-|H_i(z)|^2)|H_{j-2}(z)|^2,& if $j\ne i+2$.\cr
            }
$$            
\end{enumerate}
Let $k=1,\dots,q-1$ be fixed. From Lemma~\ref{lemma:F_k}, we see that, for $i=1,\dots,q-k$, the $i$-th row of $F_k(\Phi_{\tt H}(z))\in \mathcal{M}_{q-k,q+1}(z)$ has the following form:
\begin{enumerate}[(i)]
\item the first entry is $H_{k-1}(z)\overline{H_{i+k-1}(z)}$,
\item the $j$-th entry, for $j=2,\dots,q+1$, is 
$$\cases{  -(1-|H_{j-2}(z)|^2)H_{k-1}(z)\overline{H_{i+k-1}(z)},& if $j=k+1$ or $j=i+k+1$,\cr
               H_{k-1}(z)\overline{H_{i+k-1}(z)}|H_{j-2}(z)|^2,& if $j\ne k+1$ and $j\ne i+k+1$.\cr
            }
$$            
\end{enumerate}
Thus, we have, for each $k= 0,1,\dots,q-1$,
$$F_k(\Phi_{\tt H}(z))\left[\begin{array}{c} 
2-{\tt H}^\ast(z){\tt H}(z)\\ 
1\\
\vdots\\
1
\end{array}\right]=0,$$
where the case $k=0$ is due to the observation from Lemma~\ref{lemma:F_0}, while the case $k\ne0$ follows from Lemma~\ref{lemma:F_k}.
Hence, the Laurent polynomials $b_0(z)=2-{\tt H}^\ast(z){\tt H}(z)$, $b_1(z)=\cdots=b_q(z)=1$ satisfy (\ref{eq:FPhi}). Noting that they satisfy the other condition (\ref{eq:inhomo}) as well, we conclude that $b_0(z)=2-{\tt H}^\ast(z){\tt H}(z)$, $b_1(z)=\cdots=b_q(z)=1$ satisfy (\ref{eq:KOPT}) as desired.
\qquad\endproof

Although Theorem~\ref{thm:N(z)} provides a sufficient condition for the diagonal matrix $B(z)$ to satisfy the identity in (\ref{eqn:paraunitary}), it is easy to see that it is not necessary. For example, consider the case where ${\tt H}(z)=[0,1]^T$, and 
$$\Phi_{\tt H}(z)=\left[\begin{array}{ccc} 
0&1&0\\ 
1&0&0
\end{array}\right].$$
We notice that if we take $B(z)={\rm diag\,}([1,1,0])$, then we get the desired property, $\Phi_{\tt H}(z)B(z)\Phi_{\tt H}^\ast(z)={\tt I}$, but this $B(z)$ is not of the form in Theorem~\ref{thm:N(z)}. We also notice that the LP$^2$ matrix $\Phi_{\tt H}(z)$ in this case is actually paraunitary, hence another possible choice for $B(z)$ is ${\tt I}={\rm diag\,}([1,1,1])$, which is of the form in Theorem~\ref{thm:N(z)}. In fact, any $B(z)$ of the form $B(z)={\rm diag\,}([1,1,c])$, $c\in\RR$, satisfies $\Phi_{\tt H}(z)B(z)\Phi_{\tt H}^\ast(z)={\tt I}$ for this example. 

The next result offers conditions on ${\tt H}(z)$ under which the diagonal matrix  $B(z)$ is unique, and hence has the form given in Theorem~\ref{thm:N(z)}. 

\begin{theorem}
\label{thm:Sq}
Let ${\tt H}(z)=[H_0(z),H_1(z),\dots,H_{q-1}(z)]^T\in\mathcal{M}_{q}(z)$, and let $\Phi_{\tt H}(z)$ be the associated LP$^2$ matrix. Suppose that $B(z)\in\mathcal{M}_{(q+1)\times (q+1)}(z)$ is diagonal satisfying $\Phi_{\tt H}(z)B(z)\Phi_{\tt H}^\ast(z)={\tt I}$. Then $B(z)={\rm diag }([2-{\tt H}^\ast(z){\tt H}(z), 1, \dots, 1])$ for $z\in\Td\bks S_{{\tt H}}$, where the set $S_{{\tt H}}\subset \Td$ is defined as 
$$S_{{\tt H}}:=\{z\in\Td: H_0(z)\overline{H_1(z)}=0 \hbox{\,\,or\,\,} 1-|H_0(z)|^2-|H_1(z)|^2=0 \}$$
if $q=2$, and as
$$S_{{\tt H}}:=\{z\in\Td: H_{k-1}(z)\overline{H_{i+k-1}(z)}= 0, \hbox{\,\,for some $k=1,\dots,q-1$, $i=1,\dots,q-k$}\}$$ 
if $q\ge 3$. 
\qquad\endproof
\end{theorem}

The proof of Theorem~\ref{thm:Sq} is based on the following lemma, which computes essentially the rank of the matrix $F(\Phi_{\tt H}(z))$ using  the underlying structure of $\Phi_{\tt H}(z)$.   

\begin{lemma}
\label{lemma:split}
Let ${\tt H}(z)=[H_0(z),H_1(z),\dots,H_{q-1}(z)]^T\in\mathcal{M}_{q}(z)$, and let $\Phi_{\tt H}(z)$ be the associated LP$^2$ matrix. Let
$$\widetilde{F}(\Phi_{\tt H}(z)):=\left[\begin{array}{c} 
F_1(\Phi_{\tt H}(z))\\
\vdots\\
F_{q-1}(\Phi_{\tt H}(z))
\end{array}\right]\in \mathcal{M}_{(q-1)q/2,q+1}(z).$$
Then, $\widetilde{F}(\Phi_{\tt H}(z))$ can be factored into $D(z) A(z) U(z)$ where 
\begin{enumerate}[(i)]
\item $D(z)=D_1(z)\oplus\dots\oplus D_{q-1}(z)\in \mathcal{M}_{(q-1)q/2,(q-1)q/2}(z)$ is a diagonal matrix such that $D_k(z)={\rm diag }([H_{k-1}(z)\overline{H_k(z)},\dots,H_{k-1}(z)\overline{H_{q-1}(z)}])\in\mathcal{M}_{q-k,q-k}(z)$, for $k=1,\dots,q-1$.
\item $A(z):=A\in \mathcal{M}_{(q-1)q/2,q+1}(z)$ is a scalar matrix (i.e. independent of the variable $z$) of the form
$$A=\left[\begin{array}{c} 
A_1\\
\vdots\\
A_{q-1}
\end{array}\right]$$
with $A_k\in \mathcal{M}_{q-k,q+1}(z)$, $k=1,\dots,q-1$, being the scalar matrices whose $(i,j)$ entry, for $i=1,\dots,q-k$ and $j=1,\dots,q+1$, is 
$$
\cases{1,& if $j=1$ or $j=k+1$ or $j=i+k+1$, \cr
0,& if $j\ge 2$ and $j\ne k+1$ and $j\ne i+k+1$.
}
$$
\item
$U(z)\in \mathcal{M}_{q+1,q+1}(z)$ is an upper triangular matrix of the form 
$$
\left[\begin{array}{ccccc} 
1&|H_0(z)|^2&|H_1(z)|^2&\dots &|H_{q-1}(z)|^2\\
0&-1&0&0&0\\
0&0&-1&0&0\\
0&0&0&\ddots&0\\
0&0&0&0&-1
\end{array}\right]
$$
\end{enumerate}
Furthermore, for $z\in\Td\bks S_{{\tt H}}$, the rank of $\widetilde{F}(\Phi_{\tt H}(z))$  is $1$ if $q=2$, and $q$ if $q\ge 3$.
\qquad\endproof
\end{lemma}

\begin{proof}
Let $k=1,\dots,q-1$ be fixed. From Lemma~\ref{lemma:F_k}, we see that, for $i=1,\dots,q-k$, the $i$-th row of $F_k(\Phi_{\tt H}(z))\in \mathcal{M}_{q-k,q+1}(z)$ has the following form:
\begin{enumerate}[(i)]
\item the first entry is $H_{k-1}(z)\overline{H_{i+k-1}(z)}$,
\item the $j$-th entry, for $j=2,\dots,q+1$, is 
$$\cases{  -(1-|H_{j-2}(z)|^2)H_{k-1}(z)\overline{H_{i+k-1}(z)},& if $j=k+1$ or $j=i+k+1$,\cr
               H_{k-1}(z)\overline{H_{i+k-1}(z)}|H_{j-2}(z)|^2,& if $j\ne k+1$ and $j\ne i+k+1$.\cr
            }
$$            
\end{enumerate}
Thus, the $i$-th row of $F_k(\Phi_{\tt H}(z))$ has a common factor $H_{k-1}(z)\overline{H_{i+k-1}(z)}$, and after factoring this term out, the first entry of the remaining row vector is 1, and the $j$-th entry, for $j=2,\dots,q+1$, is
$$\cases{  |H_{j-2}(z)|^2-1,& if $j= k+1$ or $j=i+k+1$,\cr
|H_{j-2}(z)|^2,& if $j\ne k+1$ and $j\ne i+k+1$.\cr       
            }
$$           
Let $R_k(z)\in\mathcal{M}_{q-k,q+1}(z)$ be the matrix whose $i$-th row is this row vector. Thus 
the above argument gives the factorization $F_k(\Phi_{\tt H}(z))=D_k(z)R_k(z)$. Since this holds true for any fixed $k=1,\dots,q-1$, we see that $\widetilde{F}(\Phi_{\tt H}(z))=D(z)R(z)$ where
$$R(z):=\left[\begin{array}{c} 
R_1(z)\\
\vdots\\
R_{q-1}(z)
\end{array}\right]\in \mathcal{M}_{(q-1)q/2,q+1}(z).$$
Since it is easy to see $R(z)=A(z)U(z)$, we get
$\widetilde{F}(\Phi_{\tt H}(z))=D(z)A(z)U(z)$ as desired.

It is clear that, when $z\in\Td\bks S_{{\tt H}}$, the rank of $\widetilde{F}(\Phi_{\tt H}(z))$ is equal to the rank of  $A(z)=A$, which we now compute. Since $A=[1, 1, 1]^{T}$ for $q=2$, its rank is $1$. We next show that $A$ has rank $q$, for every $q\ge 3$. We note that the first column of $A$ is the vector of all $1$'s, and the sum of the other columns of $A$ is twice its first column, i.e. the vector of all $2$'s. Let $\widetilde{A}$ be the ${q(q-1)\over 2}\times q$ submatrix of $A$ obtained by removing the first column of $A$. Then ${\rm rank}\, A= {\rm rank}\,\widetilde{A}$, hence it suffices to show that ${\rm rank}\,\widetilde{A}=q$.

Since ${q(q-1)\over 2}\ge q$ for $q\ge 3$, we have ${\rm rank}\, \widetilde{A}\le q$. For the other direction, let $v_1,\dots,v_q$ be the columns of $\widetilde{A}$. Then it is easy to see that the first $q-1$ rows of $[v_2, \dots, v_q]$ is ${\tt I}_{q-1}$, the $(q-1)\times (q-1)$ identity matrix. Hence, ${\rm rank}\, \widetilde{A}\ge {\rm rank}\, {\tt I}_{q-1}=q-1$, and $v_2, \dots, v_q$ are linearly independent. Now, in order to show that ${\rm rank}\, \widetilde{A}=q$, it suffices to show that $v_1$ cannot be written as a linear combination of $v_2, \dots, v_q$. Suppose not, i.e. suppose that there exist scalar values $c_2, \dots, c_q$ such that $v_1=c_2v_2+\dots+c_qv_q$.  
Then, since $v_1$'s first $q-1$ entries are all one, using the observation that the first $q-1$ rows of the matrix $[v_2, \dots, v_q]$ is ${\tt I}_{q-1}$ again, we see that $c_2=\dots=c_q=1$. But this leads to the contradiction, since the $q$-th entry of $c_2v_2+\dots+c_qv_q=v_2+\dots+v_q$ is two, whereas the $q$-th entry of $v_1$ is zero. Therefore, we conclude that $v_1$ cannot be spanned by the vectors $v_2,\dots,v_q$, which in turn implies that ${\rm rank}\, \widetilde{A}=q$, as desired.
\end{proof}

We now present our proof of Theorem~\ref{thm:Sq}.

{\em Proof of Theorem~\ref{thm:Sq}.}

(Case I: When $q=2$). Suppose that $\Phi_{\tt H}(z){\rm diag}[b_0(z),b_1(z),b_2(z)]\Phi_{\tt H}^\ast(z)={\tt I}$, and that $z\in\Td\bks S_{\tt H}$, i.e. $z$ satisfies $H_0(z)\overline{H_1(z)}\ne 0$ and $1-|H_0(z)|^2-|H_1(z)|^2\ne 0$. Then, $b_0(z), b_1(z), b_2(z)$ satisfy (\ref{eq:inhomo}) and (\ref{eq:FPhi}) with $q=2$. Since $F(\Phi_{\tt H}(z))$, in this case, is
$$
\left[\begin{array}{ccc} 
|H_0|^2-|H_1|^2 & (1-|H_0|^2)^2-|H_1|^2|H_0|^2 & |H_0|^2|H_1|^2-(1-|H_1|^2)^2\\ 
H_0\overline{H_1} & H_0\overline{H_1}(|H_0|^2-1) & H_0\overline{H_1}(|H_1|^2-1)
\end{array}\right], 
$$
using the elementary row operations, $F(\Phi_{\tt H}(z))$ can be decomposed into
\begin{eqnarray*}
&&\left[\begin{array}{cc} 
0&1\\
1&0\end{array}\right]
\left[\begin{array}{cc} 
H_0\overline{H_1}&0\\
0&1\end{array}\right]
\left[\begin{array}{cc} 
1&0\\
|H_0|^2-|H_1|^2&1\end{array}\right]\\
&{}&\quad\quad\left[\begin{array}{cc} 
1&0\\
0&1-|H_0|^2-|H_1|^2\end{array}\right]
\left[\begin{array}{cc} 
1&-(1-|H_0|^2)\\
0&1\end{array}\right]
\left[\begin{array}{ccc} 
1&0&-(2-|H_0|^2-|H_1|^2)\\
0&1&-1\end{array}\right],
\end{eqnarray*}
where the variable $z$ is suppressed for a moment for simplicity.
Since $z\in\Td\bks S_{\tt H}$, the matrix $F(\Phi_{\tt H}(z))$ is row equivalent to the matrix
$$\left[\begin{array}{ccc} 
1&0&-(2-|H_0(z)|^2-|H_1(z)|^2)\\
0&1&-1\end{array}\right],$$
and from (\ref{eq:FPhi}) (for $q:=2$), we obtain
$$b_0(z)=(2-|H_0(z)|^2-|H_1(z)|^2)b_2(z)=(2-{\tt H}^\ast(z){\tt H}(z))b_2(z),\quad b_1(z)=b_2(z).$$
In order for $b_0(z), b_1(z), b_2(z)$ to satisfy (\ref{eq:inhomo}) (for $q:=2$) as well, $b_2(z)$ has to be $1$, hence we get
$$b_0(z)=2-{\tt H}^\ast(z){\tt H}(z),\quad b_1(z)=b_2(z)=1,\quad \hbox{ for } z\in\Td\bks S_{\tt H}$$
as desired for the case $q=2$.

(Case II: When $q\ge 3$). Suppose that $\Phi_{\tt H}(z){\rm diag}[b_0(z),b_1(z),\dots,b_q(z)]\Phi_{\tt H}^\ast(z)={\tt I}$ and that $z\in\Td\bks S_{\tt H}$, i.e. $z$ satisfies $H_{k-1}(z)\overline{H_{i+k-1}(z)}\ne 0$, for all $k=1,\dots,q-1$ and $i=1,\dots,q-k$. Then we have ${\rm rank\,}F(\Phi_{\tt H}(z))\ge {\rm rank\,} \widetilde{F}(\Phi_{\tt H}(z))=q$ from Lemma~\ref{lemma:split}.

Since we know from Theorem~\ref{thm:N(z)} that 
\begin{equation}
\label{eq:solntoFb}
b_0(z)=2-{\tt H}^\ast(z){\tt H}(z),\quad b_1(z)=\cdots=b_q(z)=1
\end{equation}
is a solution to $F(\Phi_{\tt H}(z))[b_0(z),b_1(z),\dots,b_q(z)]^T=0$, we see that ${\rm dim (ker\,}F(\Phi_{\tt H}(z)))\ge 1$. Thus, we have $q\le {\rm rank\,}F(\Phi_{\tt H}(z))=q+1-{\rm dim (ker\,}F(\Phi_{\tt H}(z)))\le q$, hence ${\rm rank\,}F(\Phi_{\tt H}(z))=q$. This means that, for each fixed $z\in\Td\bks S_{\tt H}$, there is no other solution to $F(\Phi_{\tt H}(z))[b_0(z),b_1(z),\dots,b_q(z)]^T=0$ than the constant multiples of the one in (\ref{eq:solntoFb}). In order for $b_0(z), b_1(z), \dots, b_q(z)$ to satisfy (\ref{eq:inhomo}) as well, the constant multiple has to be $1$, hence the only solution to $\Phi_{\tt H}(z){\rm diag}[b_0(z),b_1(z),\dots,b_q(z)]\Phi_{\tt H}^\ast(z)={\tt I}$ is the one given in (\ref{eq:solntoFb}) for $z\in\Td\bks S_{\tt H}$, as desired.
\qquad\endproof

Note that the set $S_{{\tt H}}\subset \Td$ could be the empty set or the entire set $\Td$. For example,  if ${\tt H}(z)=[0,1]^T$ then $S_{{\tt H}}=\TT$ and, as such, $B(z)$ does not have to take the form as in Theorem~\ref{thm:N(z)}. On the other hand, if ${\tt H}(z)=[0.5,0.5]^T$ then $S_{{\tt H}}=\emptyset$ and so $B(z)$ is uniquely determined as the one given in Theorem~\ref{thm:N(z)}.

\section{New Construction of Tight Wavelet Frames}
\label{S:wavelet}

We now use the results of the previous section  to provide a new method for constructing tight wavelet filter banks for any spatial dimension and for any dilation (or sampling) matrix. Even though many methods for constructing wavelets have been developed, methods for constructing wavelets in this generality have been scarce at best. But first,  we present a brief review on wavelets, wavelet filter banks and their polyphase representations. More details can be found, for example, in \cite{V,DV2,Hur}. 

Let $\dil$ be an $\dm\times\dm$ integer sampling or dilation matrix, and let $q:=|\det \dil|\ge 2$. We use $\capset$ (resp. $\dualset$) to denote a complete set of representatives of the distinct cosets of the quotient group $\Zd/\dil\Zd$ (resp. $2\pi(((\dil^\ast)^{-1}\Zd)/\Zd)$) containing $0$. Then the cardinality of $\capset$ (resp. $\dualset$) is $q$. We denote the elements of $\capset$ by $\nu_0:=0,\nu_1,\dots,\nu_{q-1}$. 

In the sequel we consider only FIR filters. A filter $h:\Zd\to\RR$ is called {\it lowpass} if 
$
\sum_{k\in\Zd}h(k)=\sqrt{q},
$
and {\it highpass} if
$
\sum_{k\in\Zd}h(k)=0.
$
The $z$-transform of a filter $h$ is defined as $H(z):=\sum_{k\in\Zd}h(k)z^{-k}$.
A Laurent polynomial column vector ${\tt H}(z)\in \mathcal{M}_q(z)$ is called the (synthesis) {\it polyphase representation} of a filter $h$ if 
$${\tt H}(z)=[H_{\nu_0}(z),H_{\nu_1}(z),\dots,H_{\nu_{q-1}}(z)]^T,$$ 
where $H_\nu(z)$ is the $z$-transform of the filter $h_\nu$ defined as $h_\nu(k)=h(\dil k+\nu)$, $k\in\Zd$.
Then we have $H(z)=\sum_{\nu\in\capset}z^{-\nu} H_{\nu}(z^{\dil}).$
A Laurent polynomial row vector can be associated with the analysis polyphase representation of a filter in a similar fashion. 
Since $\sum_{\nu\in\capset}H_{\nu}({\tt 1})=H({\tt 1})=\sum_{k\in\Zd}h(k)$, where ${\tt 1}=[1,\dots,1]^T\in \RR^q$ is the vector of 1's, $H({\tt 1})$  can be used to determine whether $h$ is lowpass or highpass. 

For a lowpass filter $h$, the associated refinement mask $\tau$ is defined as
$
\tau(\ome):={1\over \sqrt{q}}\sum_{k\in\Zd}h(k)e^{-ik\cdot\ome}.
$
Then $\tau$ is a Laurent trigonometric polynomial, $\tau(0)=1$, and  
\be
\label{eq:maskH}
\tau(\ome)={1\over \sqrt{q}}H(e^{i\ome})={1\over \sqrt{q}}\sum_{\nu\in\capset}e^{-i\nu\cdot\ome}H_\nu(e^{i\dil^\ast\ome}),\quad \ome\in [-\pi,\pi]^\dm.
\ee 
A refinement mask $\tau$ (or the associated filter $h$, or the polyphase representation ${\tt H}(z)$) satisfies the {\it accuracy conditions of order} $N\in\NN_0$ if
\be
\label{eq:aomask}
\tau\hbox{ has a zero of order $N$ at each }
  \gam\in\dualset\bks0.
\ee
It has {\it positive accuracy} if it satisfies the accuracy conditions of order at least one. Thus the lowpass filter $h$ has positive accuracy if and only if $H_\nu({\tt 1})=1/\sqrt{q}$, $\nu\in\capset$. 

A function $\phi\in L_2(\Rd)$ is called {\it refinable} if 
$
\hatphi(\dil^\ast\cdot)=\tau\hatphi$,
where, for $f\in L_1(\Rd)\cap L_2(\Rd)$, $\hatf(\ome):=\int_{\Rd}f(y)e^{-iy\cdot\ome}dy$.
Since, for a given refinement mask $\tau$, there exists a unique compactly supported distribution $\phi$ satisfying this refinability condition, subject to the condition $\hatphi(0)=1$ \cite{CDM}, 
we assume that $\hatphi(0)=1$. 
The order of accuracy conditions of $\tau$ is equivalent to the Strang-Fix (SF) order of the associated refinable function $\phi$ if $\phi$ is a stable $L_2(\Rd)$-function. A compactly supported function $\phi\in L_2(\Rd)$
is stable if $\hatphi$ does not have a $2\pi$-periodic zero in $\Rd$. 
Refinement masks/refinable functions play an important role in wavelet construction under the Multiresolution analysis setting \cite{Ma1}. 

\subsection{New methodology for constructing tight wavelet filter banks}
\label{subS:newFB}
Let $h$ be a lowpass filter with positive accuracy, and let ${\tt H}(z)\in\mathcal{M}_q(z)$ be its polyphase representation. Suppose that there exists a Laurent polynomial $m_{\tt H}(z)$ such that $2-{\tt H}^\ast(z){\tt H}(z)=|m_{\tt H}(z)|^2$. Then, by Theorem~\ref{thm:N(z)}  we see that 
$$\Phi_{\tt H}(z){\rm diag }([m_{\tt H}(z), 1, \dots, 1])
=\left[\begin{array}{cc} 
m_{\tt H}(z){\tt H}(z)& {\tt I}- {\tt H}(z){\tt H}^\ast(z)\\
\end{array}\right]$$
is paraunitary, i.e. $\Phi_{\tt H}(z)$ is scalable (cf. Section~\ref{S:Intro}).

As discussed in Section~\ref{S:Intro}, the LP$^2$ matrix $\Phi_{\tt H}(z)$ is paraunitary if and only if ${\tt H}^\ast(z){\tt H}(z)= 1$, $\forall z\in\Td$. Therefore, when $\Phi_{\tt H}(z)$ itself is not paraunitary, scaling it as above can  result in transforming a non-paraunitary matrix $\Phi_{\tt H}(z)$ into a paraunitary matrix $\Phi_{\tt H}(z) {\rm diag }([m_{\tt H}(z), 1, \dots, 1])$. In fact, such a scaling is special in the sense that it modifies only the first column of $\Phi_{\tt H}(z)$, from ${\tt H}(z)$ to $m_{\tt H}(z){\tt H}(z)$, while keeping all the other columns intact.

A key assumption in the above approach is the existence of a Laurent polynomial $m_{\tt H}(z)$ such that $2-{\tt H}^\ast(z){\tt H}(z)=|m_{\tt H}(z)|^2$. For such a factorization to exist, it is necessary that $2-{\tt H}^\ast(z){\tt H}(z)\ge 0$, for all $z\in\Td$.
Since ${\tt H}^\ast(z){\tt H}(z)=\sum_{\nu\in\capset}|H_\nu(z)|^2$, where ${\tt H}(z)=[H_{\nu_0}(z),H_{\nu_1}(z),\dots,H_{\nu_{q-1}}(z)]^T$, in order to check the condition ${\tt H}^\ast(z){\tt H}(z)\le 2, \forall z\in\Td$, it suffices to  bound $H_\nu(z)$ for each $\nu\in\capset$. 
In some cases, it might be easier to deduce this  from $\tau$  (cf. Example 2 in Section~\ref{subS:example}), as illustrated by  the following result. 

\begin{lemma}
\label{lemma:HstarH}
${\tt H}^\ast(e^{i\dil^\ast\ome}){\tt H}(e^{i\dil^\ast\ome})=\sum_{\gam\in\dualset}|\tau(\ome+\gam)|^2$, for all $\ome\in[-\pi,\pi]^\dm.$
\qquad\endproof
\end{lemma}

\begin{proof}
It is easy to observe that, for all $\nu\in\capset$ and for all $\ome\in[-\pi,\pi]^\dm$,
$H_\nu(e^{i\dil^\ast\ome})={1\over q}\sum_{\gam\in\Gam^\ast}e^{i(\ome+\gam)\cdot\nu} H(e^{i(\ome+\gam)}).$
Together with ${\tt H}^\ast(z){\tt H}(z)=\sum_{\nu\in\capset}|H_\nu(z)|^2$, we obtain the desired identity 
\beaN
&{}&{\tt H}^\ast(e^{i\dil^\ast\ome}){\tt H}(e^{i\dil^\ast\ome})=\sum_{\nu\in\capset}|H_\nu(e^{i\dil^\ast\ome})|^2\\
&=&\sum_{\nu\in\capset}\left({1\over q}\sum_{\gam\in\Gam^\ast}e^{i(\ome+\gam)\cdot\nu} H(e^{i(\ome+\gam)})\right)\left({1\over q}\sum_{\tilde\gam\in\Gam^\ast}e^{-i(\ome+\tilde\gam)\cdot\nu} \overline{H(e^{i(\ome+\tilde\gam)})}\right)\\
&=&\sum_{\nu\in\capset}{1\over q^2}\sum_{\gam\in\Gam^\ast}\sum_{\tilde\gam\in\Gam^\ast}e^{i(\gam-\tilde\gam)\cdot\nu} H(e^{i(\ome+\gam)})\overline{H(e^{i(\ome+\tilde\gam)})}\\
&=&{1\over q^2}\sum_{\gam\in\Gam^\ast}\sum_{\tilde\gam\in\Gam^\ast}\left(\sum_{\nu\in\capset}e^{i(\gam-\tilde\gam)\cdot\nu} \right)H(e^{i(\ome+\gam)})\overline{H(e^{i(\ome+\tilde\gam)})}
=\sum_{\gam\in\Gam^\ast}|\tau(\ome+\gam)|^2,
\eeaN
where the relation between $\tau(\ome)$ and $H(e^{i\ome})$ (cf. (\ref{eq:maskH})) and the following identity 
$$
\sum_{\nu\in\capset}(e^{i\gam})^\nu
  =\cases{q,& if $\gam=0$,\cr
          0,             & if $\gam\in\capset^\ast\bks\{0\}$,\cr}
$$
are used for the last equality.
\end{proof}

The factorization of $2-{\tt H}^\ast(z){\tt H}(z)$ can be dealt with ease for $1$-D case by using the  well-known Fej\'er-Riesz lemma.

\begin{lemma}[Fej\'er-Riesz Lemma, \cite{Fejer,Riesz}]
\label{lemma:FR}
Suppose $P(z)=\sum_{k=-r}^r p(k) z^{-k}\ge 0$, for all $z\in\TT$. Then there exists a $1$-D Laurent polynomial $Q(z)=\sum_{k=0}^rq(k) z^{-k}$ such that $P(z)=|Q(z)|^2, \forall z\in\TT.$ 
\qquad\endproof
\end{lemma}

Using this result, we obtain the following new method for constructing $1$-D tight wavelet filter banks. 

\begin{theorem}
\label{thm:construction_oneD}
Let $h$ be a $1$-D lowpass filter with positive accuracy and dilation $\lambda\ge 2$, and let ${\tt H}(z)$ be its polyphase representation. Suppose $2-{\tt H}^\ast(z){\tt H}(z)>0$, $\forall z\in\TT$. Then there is a polynomial $m_{\tt H}(z)$ such that $[m_{\tt H}(z){\tt H}(z), {\tt I}- {\tt H}(z){\tt H}^\ast(z)]$
gives rise to a tight wavelet filter bank whose lowpass filter $\widetilde{h}$  is associated with $m_{\tt H}(z){\tt H}(z)$ and has the same accuracy as $h$. Furthermore, if the support of $h$ is contained in $\{0,1,\dots,s\}$, then the support of $\widetilde{h}$ is contained in $\{0,1,\dots, 2s\}$.
\qquad\endproof
\end{theorem}

\begin{proof}
Since $2-{\tt H}^\ast(z){\tt H}(z)>0, \forall z\in\TT$, by Lemma~\ref{lemma:FR}, there exists an $m_{\tt H}(z)$ such that $2-{\tt H}^\ast(z){\tt H}(z)=|m_{\tt H}(z)|^2$, $\forall z\in\TT$. Thus, $[m_{\tt H}(z){\tt H}(z), {\tt I}- {\tt H}(z){\tt H}^\ast(z)]$ is paraunitary, i.e. it gives rise to a tight filter bank. Furthermore, observe that the sum of the $j$-th column of ${\tt I}- {\tt H}(z){\tt H}^\ast(z)$ evaluated at $z=1$ is equal to 
\be
\label{eq:colofIminusHHstar}
1-\left(\sum_{\nu\in\capset}H_\nu(1)\right)\overline{H_{\nu_{j-1}}(1)}.
\ee
Using the fact that $h$ is lowpass with positive accuracy, which implies $\sum_{\nu\in\capset}H_\nu(1)=\sqrt{\lambda}$ and $H_{\nu_{j-1}}(1)=1/\sqrt{\lambda}$, we get that the number in (\ref{eq:colofIminusHHstar}) is equal to zero for each $j=1,\dots,\lambda$, which means that the filters associated with the columns of ${\tt I}- {\tt H}(z){\tt H}^\ast(z)$ are all highpass. Hence our tight filter bank is actually a wavelet tight filter bank.

Let $\tau$ and $\widetilde{\tau}$ be the refinement masks associated with ${\tt H}(z)$ and $m_{\tt H}(z){\tt H}(z)$, respectively. Then  (\ref{eq:maskH}) implies
$$
\widetilde\tau(\ome)={1\over \sqrt{\lambda}}\sum_{\nu\in\capset}e^{-i\ome\cdot\nu}m_{\tt H}(e^{i\lambda\ome})H_\nu(e^{i\lambda\ome})=m_{\tt H}(e^{i\lambda\ome})\tau(\ome).$$
Since $2-{\tt H}^\ast(e^{i\lambda\ome}){\tt H}(e^{i\lambda\ome})=|m_{\tt H}(e^{i\lambda\ome})|^2>0$, $\forall \ome\in [-\pi,\pi]$, we see that $m_{\tt H}(e^{i\lambda\cdot})$ does not vanish on $[-\pi,\pi]$. Therefore, $\tau$ and $\widetilde\tau$ have exactly the same accuracy (cf. (\ref{eq:aomask})).

To prove the result about the support of filters, we note that, since $\tau(\ome)={1\over \sqrt{\lambda}}\sum_{k=0}^{s}h(k)e^{-i k\ome}$, from Lemma~\ref{lemma:HstarH}, $2-{\tt H}^\ast(e^{i\lambda \ome}){\tt H}(e^{i\lambda \ome})=\sum_{k=-s}^s a(k) e^{-ik\ome}$. Hence, by Lemma~\ref{lemma:FR}, we see that the Laurent polynomial $m_{\tt H}(z)$ with $|m_{\tt H}(z)|^2=2-{\tt H}^\ast(z){\tt H}(z)$ can be chosen so that $m_{\tt H}(e^{i\lambda\ome})=\sum_{k=0}^sb(k)e^{-ik\ome}$. 
Therefore, the support of $\widetilde{h}$ associated with $m_{\tt H}(e^{i\lambda\ome})\tau(\ome)$ is contained in $\{0,1,\dots,2s\}$, as desired.
\end{proof}

For the multi-D case, factoring $2-{\tt H}^\ast(z){\tt H}(z)$ into $|m_{\tt H}(z)|^2$ for some Laurent polynomial $m_{\tt H}(z)$ is a nontrivial problem, which we hope to address in the future. However, once such an $m_{\tt H}(z)$ exists, by using similar arguments as in the proof of Theorem~\ref{thm:construction_oneD} (thus we omit the proof for multi-D case), we obtain the following result. 

\begin{theorem}
\label{thm:construction_nD}
Let $h$ be an $n$-D, $n\ge 2$, lowpass filter with positive accuracy, and let ${\tt H}(z)$ be its polyphase representation. 
Suppose that $2-{\tt H}^\ast(z){\tt H}(z)>0, \forall z\in\Td$ and that there exists a Laurent polynomial $m_{\tt H}(z)$ such that $2-{\tt H}^\ast(z){\tt H}(z)=|m_{\tt H}(z)|^2, \forall z\in\Td$. 
Then $[m_{\tt H}(z){\tt H}(z), {\tt I}- {\tt H}(z){\tt H}^\ast(z)]$
gives rise to a tight wavelet filter bank whose lowpass filter  is associated with $m_{\tt H}(z){\tt H}(z)$ and has the same accuracy as $h$.
\qquad\endproof
\end{theorem}

\subsection{New tight wavelet frames for $L_2(\RR)$}
\label{subS:newframe}

A tight wavelet frame for $L_2(\Rd)$ is a generalization of the orthonormal wavelet basis for $L_2(\Rd)$, and is developed out of an effort to overcome some of the limitations of the orthonormal wavelet bases. The theory, algorithms, and applications of tight wavelet frames are extensively studied in the literature (for example,  see \cite{DHRS,CPSS} and references therein). Constructing tight wavelet frames from tight wavelet filter banks is an important issue but is often not so straightforward.

In this section we consider the univariate case ($n=1$), and present a method for obtaining tight wavelet frames for $L_2(\RR)$ from our 1-D tight wavelet filter banks in the previous subsection. In particular, if we know that the refinable function associated with the new lowpass filter $\widetilde{h}$ defined in Theorem~\ref{thm:construction_oneD} is square integrable over $\RR$, then the associated  filter bank  gives rise to a tight wavelet frame. 

\begin{corollary}
\label{coro:constructionW_oneD}
Let $\phi\in L_2(\RR)$ be a stable refinable function with positive SF order and dilation $\lambda\ge 2$, and let ${\tt H}(z)$ be the associated polyphase representation. Suppose that $2-{\tt H}^\ast(z){\tt H}(z)>0$, $\forall z\in\TT$, and that the compactly supported refinable distribution $\widetilde{\phi}$ associated with $m_{\tt H}(z){\tt H}(z)$, where $|m_{\tt H}(z)|^2=2-{\tt H}^\ast(z){\tt H}(z)$, is in  $L_2(\RR)$. Then $[m_{\tt H}(z){\tt H}(z), {\tt I}- {\tt H}(z){\tt H}^\ast(z)]$ gives rise to a tight wavelet frame, with a stable refinable function $\widetilde{\phi}$ having the same SF order as $\phi$. Furthermore, if the support of $\phi$ is contained in $[0,r]$ with $r\in \NN$, then the support of $\widetilde{\phi}$ is  contained in $[0,2r]$. 
\qquad\endproof
\end{corollary}

\begin{proof}
From the assumptions that $\widetilde\phi$ is an $L_2(\RR)$-function, $\phi$ is a stable $L_2(\RR)$-function with positive SF order and with dilation factor $\lambda\ge 2$, and that $m_{\tt H}(e^{i\lambda\cdot})$ does not vanish on $[-\pi,\pi]$, we see that $\widetilde{\phi}$ is stable and has the same SF order as $\phi$.

By Theorem~\ref{thm:construction_oneD}, we see that $[m_{\tt H}(z){\tt H}(z), {\tt I}- {\tt H}(z){\tt H}^\ast(z)]$ gives rise to a tight wavelet filter bank whose lowpass filter is $\widetilde{h}$ , where $\widetilde{h}$ is associated with $\widetilde{\phi}$. Since $\widetilde\phi$ is an $L_2(\RR)$-function, this tight wavelet filter bank gives rise to a tight wavelet frame for $L_2(\RR)$.

Since the assumption that the support of $\phi$ is contained in $[0,r]$ implies that the support of its associated lowpass filter $h$ is contained in $\{0,1,\dots,(\lambda-1)r\}$, by Theorem~\ref{thm:construction_oneD}, we see that the support of $\widetilde{h}$ is contained in $\{0,1,\dots,2(\lambda-1)r\}$, which in turn gives that the support of $\widetilde{\phi}$ is contained in $[0,2r]$, which finishes the proof.
\end{proof}

One can always assume that the  refinable function $\phi$ is  ``good'' enough, so that the assumptions on $\phi$ are satisfied. On the other hand, whether $\widetilde\phi$ is in $L_2(\RR)$ or not may not be easily verifiable. In the examples below, we check that $\widetilde\phi \in L_2(\RR)$ by appealing to a general result proved in \cite[Proposition 4.5]{HR4}.  The original result was stated for $\lambda=2$ case and it can be extended to more general case for $\lambda\ge 2$ without much difficulty by following the original arguments closely, hence its proof is omitted. In the statement we use the smoothness class $\mathcal{R}^\alpha$, $\alpha>0$, which is very similar to the class of functions with H\"{o}lder exponent $\alpha$. We refer the aforementioned paper for the exact definition of this smoothness class.

\begin{theorem}[Proposition 4.5 in \cite{HR4}]
\label{thm:riesz}
Let $\lambda\ge 2$ be an integer. Let $\phi\in L_2(\RR)$ be a refinable function associated with refinement mask $\tau$ and with dilation $\lambda$, and let $\widetilde\phi$ be the compactly supported refinable distribution associated with the new refinement mask $\widetilde\tau=\xi(\lambda\cdot)\tau$ and with dilation $\lambda$. Define, for each $j\in\ZZ$ and $\eps>0$,
$$\Ome_{j,\eps}:=\{\ome+i\tet\in\CC: \ome\in \Ome_j,\ |\tet|<\eps\},\quad \hbox{where}\quad\Ome_j:=\{\ome\in\RR:\lambda^jK\le|\ome|\le\lambda^{j+1}K\},$$
with $K$ some positive number.
Let
$$\beta:=-\inf_{\eps>0}\limsup_{j\to\infty}
  {\log_\lambda\norm{\hat{\phi}}_{L_\infty(\Ome_{j,\eps})}\over j}.$$
If $\alp:=\beta-\log_\lambda\norm{\xi}_{L_\infty([-\pi,\pi])}-1>0$, then $\widetilde\phi\in \mathcal{R}^\alpha$. In particular, $\widetilde\phi\in L_2(\RR)$.
\qquad\endproof
\end{theorem}

\subsection{Examples and concluding remark}
\label{subS:example}
We now illustrate our results through some examples. 

{\noindent\bf Example 1 (1-D dyadic wavelet frames generated from Deslauriers-Dubuc functions): }
Let $\lambda=2$, and let $\phi\in L_2(\RR)$ be the Deslauriers-Dubuc (DD) interpolatory refinable function of order $2k$, supported on $[0,4k-2]$, for $k\in\NN$ \cite{DD}. Then, $\phi$ is stable with SF order $2k$, and with the choice of $\capset=\{0,1\}$ and $\dualset=\{0,\pi\}$, the associated $z$-transform and refinement mask are given as, respectively, 
$$H(z)=\sqrt{2}z^{-2k+1}\left({1\over 4}(z+2+z^{-1})\right)^kP_k\left(-{1\over 4}(z-2+z^{-1})\right),\quad z\in\TT,$$
$$\tau(\ome)=e^{-(2k-1)i\ome}\cos^{2k}({\ome\over 2})P_k(\sin^2({\ome\over 2})),\quad \ome\in[-\pi,\pi],$$
where
$$P_k(x)=\sum_{j=0}^{k-1}{\disp(k-1+j)!\over\disp j!(k-1)!}x^j,$$
and the components of the polyphase representation ${\tt H}(z)=[H_0(z),H_1(z)]^T$ satisfy
$$H_1(z)={1\over \sqrt{2}}z^{-k+1},\quad H_0(z^2)=H(z)-{1\over \sqrt{2}}z^{-2k+1}.$$

Since $z^{2k-1}H(z)+\overline{z}^{2k-1}H(\overline{z})=\sqrt{2}$, $z^{2k-1}H(z)\ge 0$, and $\overline{z}^{2k-1}H(\overline{z})\ge 0$, $\forall z\in\TT$, we see $0\le z^{2k-1}H(z)\le \sqrt{2}$, and $\left|H_0(z^2)\right|=\left|z^{2k-1}H(z)-{1/\sqrt{2}}\right|\le {1/\sqrt{2}}$, which in turn implies $|H_0(z)|\le {1/\sqrt{2}}$, $\forall z\in\TT$. Combining this with $|H_1(z)|={1/\sqrt{2}}$, we have 
\be
\label{eq:DDpositivity}
2-{\tt H}^\ast(z){\tt H}(z)=2-|H_0(z)|^2-|H_1(z)|^2={3\over 2}-|H_0(z)|^2\ge 1>0,\quad \forall z\in\TT.
\ee
Let $m_{\tt H}(z)$ be the Laurent polynomial satisfying $|m_{\tt H}(z)|^2=2-{\tt H}^\ast(z){\tt H}(z)$, whose existence is guaranteed by Lemma~\ref{lemma:FR}. Let $\widetilde\phi$ be the compactly supported refinable distribution associated with $\widetilde{\tau}=m_{\tt H}(e^{i2\cdot})\tau$. For the DD refinable function $\phi$ of order $2k$, the parameter $\beta$ in Theorem~\ref{thm:riesz} satisfies (see, for example, \cite{Da,HR4})
$$\beta\ge 2k-\log_2 P_{k}(3/4)\ge 
k(2-\log_23)+\log_23.$$
Since $|m_{\tt H}(e^{i\ome})|=\sqrt{2-{\tt H}^\ast(e^{i\ome}){\tt H}(e^{i\ome})}\le \sqrt{6}/2$, $\forall \ome\in[-\pi,\pi]$,
we see that
\begin{eqnarray*}
\alpha&\,{=}\,&\beta-\log_2\norm{m_{\tt H}(e^{i\cdot})}_{L_\infty[-\pi,\pi]}-1\\
&\ge& 2k-\log_2 P_{k}(3/4)-\log_2\sqrt{6}\ge k(2-\log_23)+{1\over 2}(\log_23-1)>0.
\end{eqnarray*}
and, by Theorem~\ref{thm:riesz}, $\widetilde\phi$ is in $L_2(\RR)$, for each $k\in\NN$. Hence by Corollary~\ref{coro:constructionW_oneD} we obtain a tight wavelet frame whose refinable function $\widetilde\phi$, where $\widetilde{\phi}$ is stable with SF order $2k$ and its support is contained in the interval $[0,8k-4]$. 

When $k=1$, the refinement mask is $\tau(\ome)=e^{-i\ome}\cos^2(\ome/2)=(1+2e^{-i\ome}+e^{-2i\ome})/4$
and the corresponding DD refinable function is the hat function: 
\bea
\label{eq:hat}
\phi(x)=\cases{x, & if $0\le x\le 1$,\cr
2-x,& if $1\le x\le 2$, \cr
0,& otherwise.
}
\eea
After applying our method as suggested above, we get the new refinement mask 
\begin{eqnarray*}
&\,{}\,&\widetilde\tau(\ome)=e^{-i\ome}\cos^2(\ome/2)\left({{2+\sqrt{6}}\over {4}}+{{2-\sqrt{6}}\over {4}}e^{-2i\ome}\right)\\
&=&{{2+\sqrt{6}}\over 16}+{{2+\sqrt{6}}\over 8}e^{-i\ome}+{1\over 4}e^{-2i\ome}+{{2-\sqrt{6}}\over 8}e^{-3i\ome}+{{2-\sqrt{6}}\over 16}e^{-4i\ome}.
\end{eqnarray*}
The new refinable function $\widetilde\phi$ associated with $\widetilde\tau$ is depicted in Fig.~\ref{figure:tildehat}, together with the original refinable function $\phi$.
\qquad\endproof

\begin{figure}
\begin{center}
\includegraphics[scale=0.405]{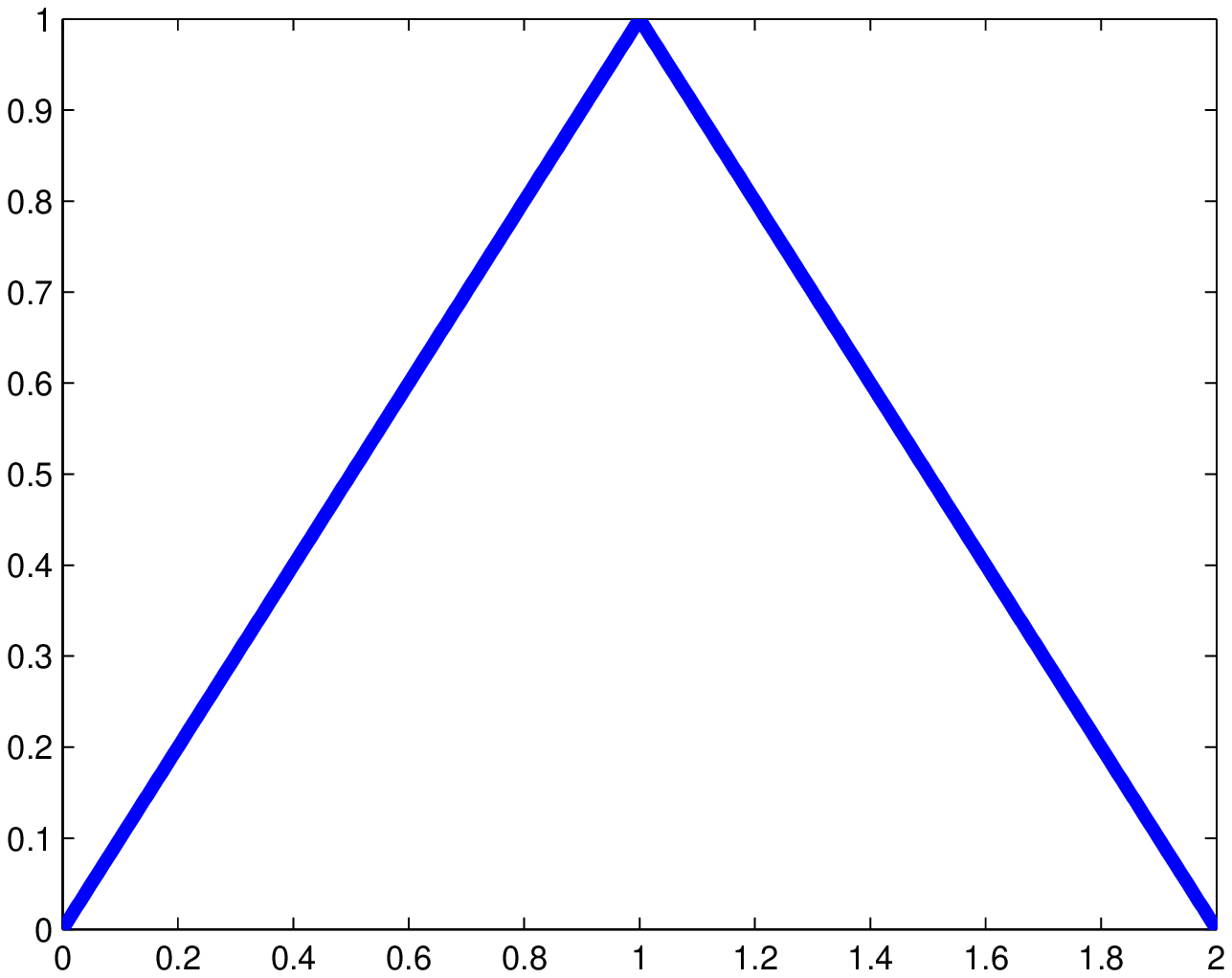}
\includegraphics[scale=0.405]{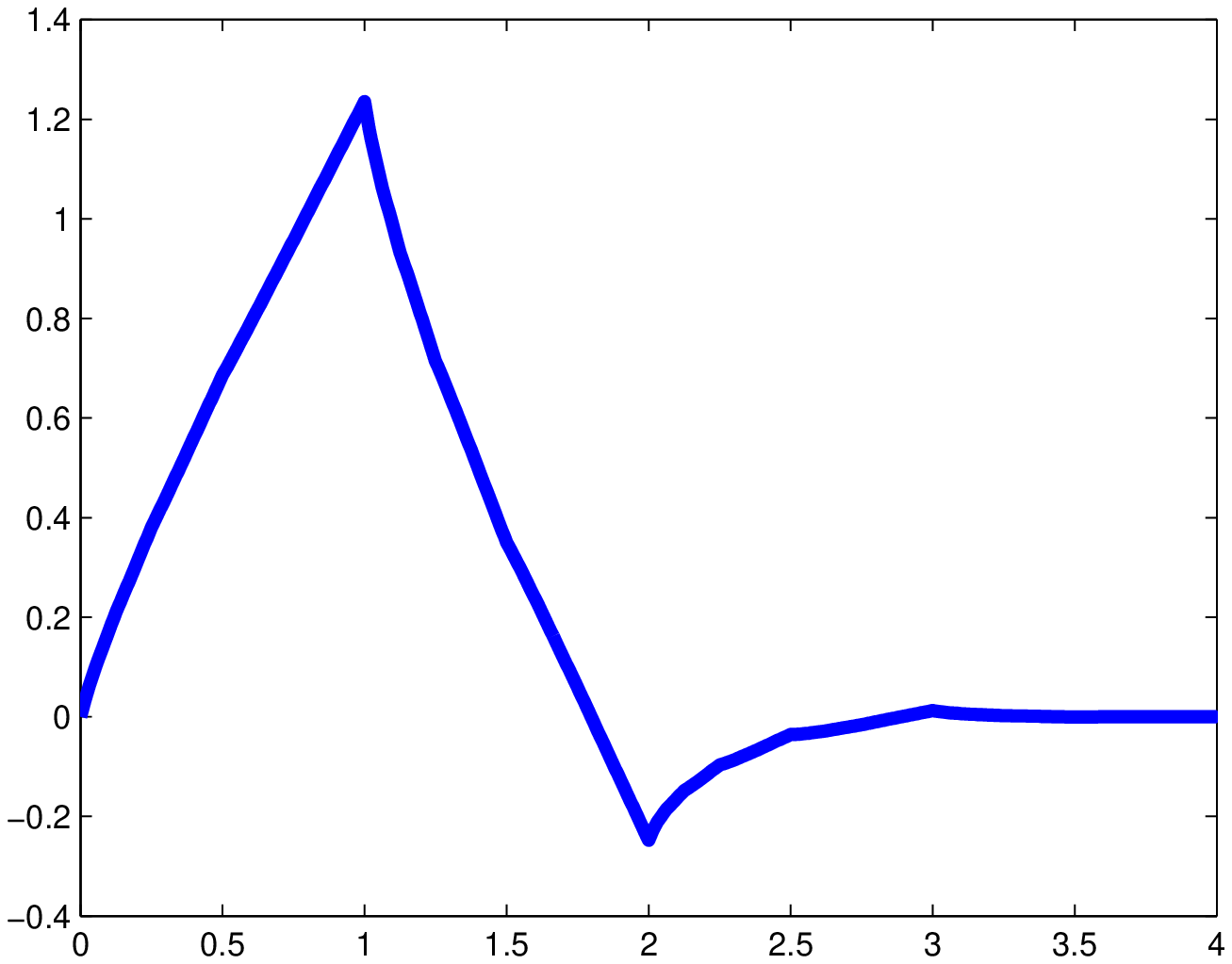}
\caption{The original ($\phi$, left) and the new ($\widetilde\phi$, right) refinable functions of Example 1 for $k=1$.}
\label{figure:tildehat}
\end{center}
\end{figure}

{\noindent\bf Example 2 (1-D dyadic wavelet frames generated from the B-splines): }
Still taking $\lambda=2$, $\capset=\{0,1\}$, and $\dualset=\{0,\pi\}$, we let $\phi\in L_2(\RR)$ be the B-spline of order $k$, supported on $[0,k]$, $k\in\NN$, whose associated refinement mask is
$$\tau(\ome)=\left({{1+e^{-i\ome}}\over 2}\right)^k,\quad \ome\in[-\pi,\pi].$$
It is well known that $\phi$ is stable with SF order $k$. 

Since 
$$|\tau(\ome)|^2+|\tau(\ome+\pi)|^2=\cos^{2k}\left({\ome\over 2}\right)+\sin^{2k}\left({\ome\over 2}\right)\le \left(\cos^{2}\left({\ome\over 2}\right)+\sin^{2}\left({\ome\over 2}\right)\right)^{k}=1<2,$$
by invoking Lemma~\ref{lemma:HstarH}, we see that $2-{\tt H}^\ast(z){\tt H}(z)>0$, $\forall z\in\TT$, is satisfied, for each $k\in\NN$, where ${\tt H}(z)$ is the associated polyphase representation. 

For the rest of this example, we assume that $k\ge 3$, since, when $k=1$, $\phi$ is the Haar refinable function which does not need any scaling to produce a tight wavelet frame, and when $k=2$, $\phi$ is the hat function that we discussed already in Example 1. 

Let $\widetilde\phi$ be the refinable distribution associated with the new refinement mask $\widetilde{\tau}(\ome)=m_{\tt H}(e^{i2\ome})\tau(\ome)$, where $|m_{\tt H}(z)|^2=2-{\tt H}^\ast(z){\tt H}(z)$, $\forall z\in\TT$. For the B-spline refinable function $\phi$ of order $k$, it is easy to see that
the parameter $\beta$ in Theorem~\ref{thm:riesz} satisfies $\beta\ge k$.
By combining this with $\norm{m_{\tt H}(e^{i\cdot})}_{L_\infty[-\pi,\pi]}\le\sqrt{2}$, we get that, by Theorem~\ref{thm:riesz}, the refinable distribution $\widetilde\phi\in \mathcal{R}^{k-1.5}$. In particular, $\widetilde{\phi}$ is in $L_2(\RR)$, for each $k\ge 3$. Thus, by Corollary~\ref{coro:constructionW_oneD}, we see that $[m_{\tt H}(z){\tt H}(z), {\tt I}- {\tt H}(z){\tt H}^\ast(z)]$ gives rise to a tight wavelet frame, whose refinable function $\widetilde{\phi}$ is stable with SF order $k$, and with its support contained in $[0,2k]$. 

When $k=3$, $\tau(\ome)=(1+e^{-i\ome})^3/8=(1+3e^{-i\ome}+3e^{-2i\ome}+e^{-3i\ome})/8$ and the corresponding refinable function $\phi$ is the cubic B-spline supported on $[0,3]$. Our scaling process produces the new refinement mask 
\begin{eqnarray*}
\widetilde{\tau}(\ome)&\,{=}\,&{1+3e^{-i\ome}+3e^{-2i\ome}+e^{-3i\ome}\over 8}\left({{2+\sqrt{7}}\over 4}+{{2-\sqrt{7}}\over 4}e^{-2i\ome}\right)\\
&=&{{2+\sqrt{7}}\over 32}+{{6+3\sqrt{7}}\over 32}e^{-i\ome}+{{8+2\sqrt{7}}\over 32}e^{-2i\ome}\\
&\quad&\,+\,{{8-2\sqrt{7}}\over 32}e^{-3i\ome}+{{6-3\sqrt{7}}\over 32}e^{-4i\ome}+{{2-\sqrt{7}}\over 32}e^{-5i\ome}
\end{eqnarray*}
and the corresponding refinable function $\widetilde{\phi}$ is depicted in Fig.~\ref{figure:tildeB3}, together with the original refinable function $\phi$.
\qquad\endproof

\begin{figure}
\begin{center}
\includegraphics[scale=0.405]{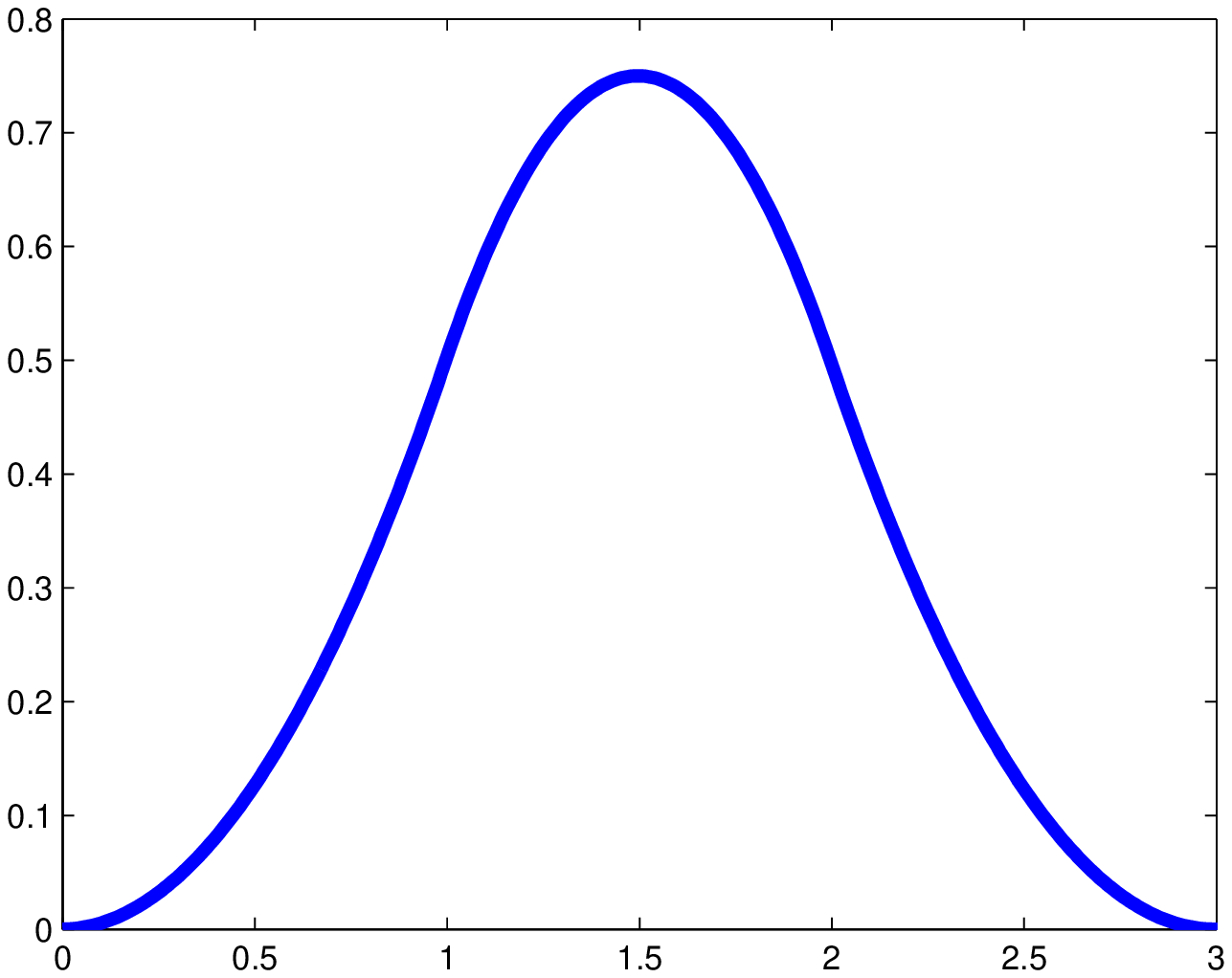}
\includegraphics[scale=0.405]{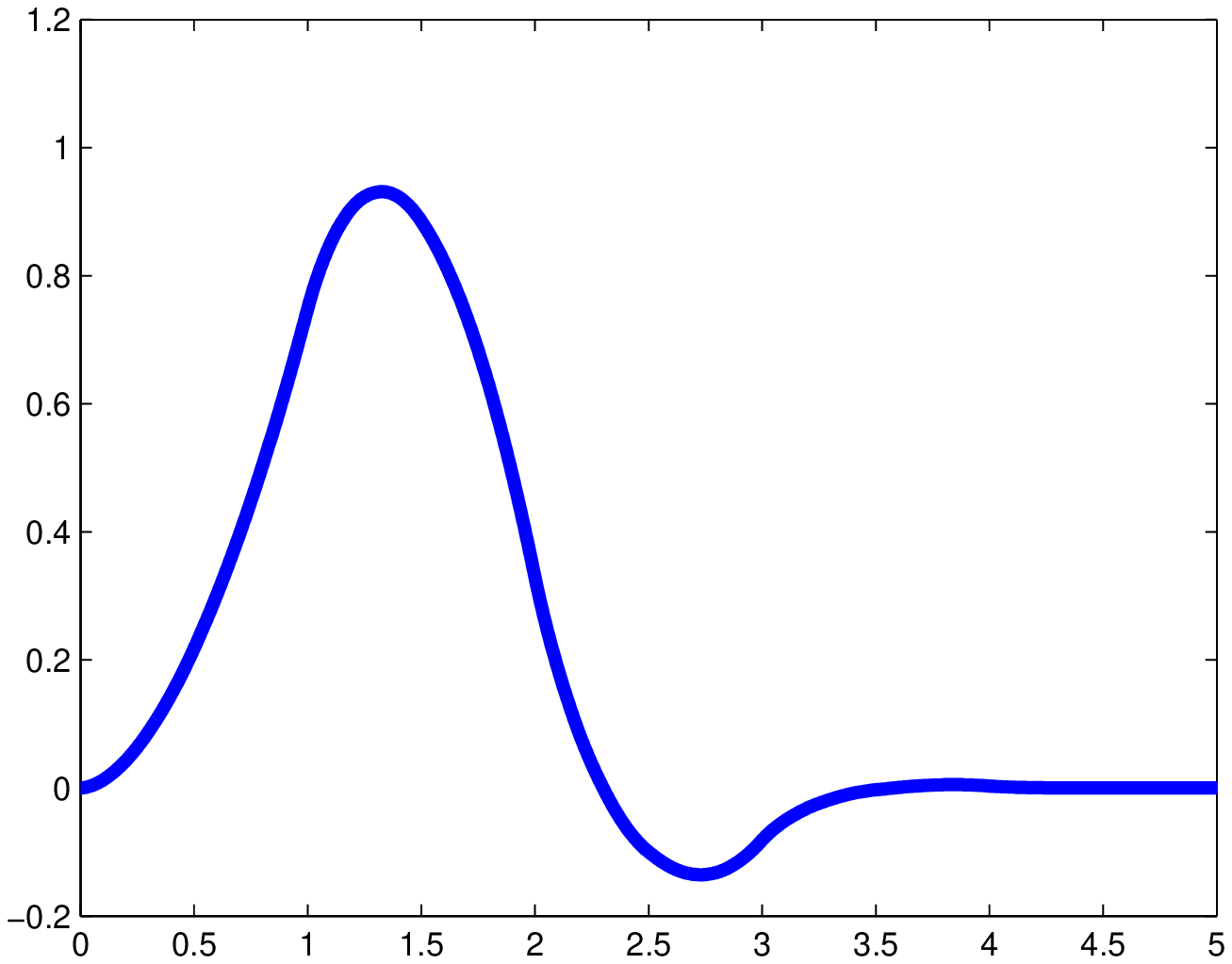}
\caption{The original ($\phi$, left) and the new ($\widetilde\phi$, right) refinable functions of Example 2 for $k=3$.}
\label{figure:tildeB3}
\end{center}
\end{figure}

{\noindent\bf Example 3 (1-D non-dyadic wavelet frames generated from the hat function):}
Let us now consider non-dyadic dilations, i.e. $\lambda\ge 3$.
We let $\phi\in L_2(\RR)$ be the hat function in (\ref{eq:hat}), which we considered in Example 1 for the dyadic case, i.e. $\lambda=2$. We know that $\phi$ is stable with SF order $2$ and is supported on $[0,2]$. Let $\lambda\ge 3$ be the 1-D non-dyadic integer dilation factor. We take $\capset=\{0,1,\dots,\lambda-1\}$, and $\dualset=\{0,{1\over\lambda}2\pi,\dots,{(\lambda-1)\over\lambda}2\pi\}$.  Let $\tau$  be the associated refinement mask with dilation $\lambda$. Then $\tau$ is given as, $\forall\ome\in[-\pi,\pi]$,
\begin{eqnarray*}
\tau(\ome)&\,{=}\,&{1\over \lambda^2}e^{-i(\lambda-1)\ome}\left(e^{i(\lambda-1)\ome}+2e^{i(\lambda-2)\ome}+\cdots+(\lambda-1)e^{i\ome}+\lambda\right.\\
&\quad&\quad\quad\quad\quad\quad\quad \left.+(\lambda-1)e^{-i\ome}+\cdots+2e^{-i(\lambda-2)\ome}+e^{-i(\lambda-1)\ome}\right).
\end{eqnarray*}

Since $\sum_{\gam\in\dualset} e^{i(\lambda-1)(\ome+\gam)}\tau(\ome+\gam)=1$ and $0\le e^{i(\lambda-1)\ome}\tau(\ome)\le1$, for all $\ome\in[-\pi,\pi]$, we have that
$|\tau(\ome+\gam)|^2\le e^{i(\lambda-1)(\ome+\gam)}\tau(\ome+\gam)$, $\forall\ome\in[-\pi,\pi]$, $\forall\gam\in\dualset$. From this and Lemma~\ref{lemma:HstarH}, the polyphase representation ${\tt H}(z)$ satisfies, for $\ome\in[-\pi,\pi]$,
$${\tt H}^\ast(e^{i\lambda\ome}){\tt H}(e^{i\lambda\ome})=\sum_{\gam\in\dualset} |\tau(\ome+\gam)|^2\le \sum_{\gam\in\dualset}e^{i(\lambda-1)(\ome+\gam)}\tau(\ome+\gam)=1<2.$$
Thus, by Lemma~\ref{lemma:FR}, there exists a Laurent polynomial $m_{\tt H}(z)$ satisfying $|m_{\tt H}(z)|^2=2-{\tt H}^\ast(z){\tt H}(z)$. Let $\widetilde{\phi}$ be the compactly supported refinable distribution associated with $\widetilde{\tau}(\ome)=m_{\tt H}(e^{i\lambda\ome})\tau(\ome)$. Using the facts that the last component of the polyphase representation ${\tt H}(z)$ is $H_{\lambda-1}(z)={1\over \sqrt{\lambda}}$ which implies that $|m_{\tt H}(e^{i\ome})|\le \sqrt{2-{1\over \lambda}}$, $\forall\ome\in[-\pi,\pi]$, and that the parameter $\beta$ of Theorem~\ref{thm:riesz} in this case satisfies $\beta\ge 2$, we have
$$\alpha=\beta-\log_\lambda\norm{m_{\tt H}(e^{i\cdot})}_{L_\infty[-\pi,\pi]}-1\ge 1-{1\over 2}\log_\lambda\left(2-{1\over \lambda}\right)\ge 1-{1\over 2}\log_3\left(2-{1\over 3}\right)>0,$$
and as a result, $\widetilde{\phi}\in \mathcal{R}^{1-{1\over 2}\log_\lambda(2-{1\over \lambda})}$, for each $\lambda\ge 3$. In particular, $\widetilde{\phi}\in L^2(\RR)$, for each $\lambda\ge 3$. Thus, by Corollary~\ref{coro:constructionW_oneD}, the tight wavelet filter bank in Theorem~\ref{thm:construction_oneD} gives rise to a 1-D tight wavelet frame associated with the stable $L^2$-function $\widetilde\phi$ of SF order $2$.

\begin{figure}
\begin{center}
\includegraphics[scale=0.405]{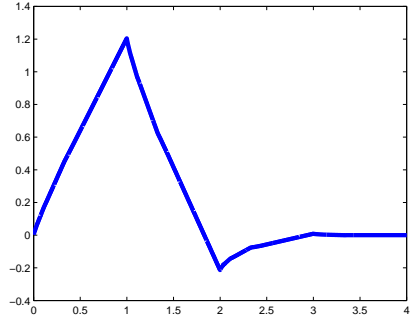}
\caption{The new refinable function $\widetilde\phi$ of Example 3 for $\lambda=3$.}
\label{figure:tildehatlam3}
\end{center}
\end{figure}

When $\lambda=3$, the new refinement mask is given as 
\begin{eqnarray*}
&\,{}\,&\widetilde{\tau}(\ome)={{3\sqrt{3}+\sqrt{43}}\over 54\sqrt{3}}+{{6\sqrt{3}+2\sqrt{43}}\over 54\sqrt{3}}e^{-i\ome}+{{9\sqrt{3}+3\sqrt{43}}\over 54\sqrt{3}}e^{-2i\ome}+{{9\sqrt{3}+\sqrt{43}}\over 54\sqrt{3}}e^{-3i\ome}\\
&&\,+\,{{9\sqrt{3}-\sqrt{43}}\over 54\sqrt{3}}e^{-4i\ome}+{{9\sqrt{3}-3\sqrt{43}}\over 54\sqrt{3}}e^{-5i\ome}+{{6\sqrt{3}-2\sqrt{43}}\over 54\sqrt{3}}e^{-6i\ome}+{{3\sqrt{3}-\sqrt{43}}\over 54\sqrt{3}}e^{-7i\ome}
\end{eqnarray*} 
and the graph of the new refinable function $\widetilde{\phi}$ is placed in Fig.~\ref{figure:tildehatlam3}. The graph of the original refinable function $\phi$ (i.e. the hat function) that gives rise to this new refinable function $\widetilde{\phi}$ can be found in Fig.~\ref{figure:tildehat}. Although the graphs of $\widetilde{\phi}$ in Fig.~\ref{figure:tildehatlam3} and ~\ref{figure:tildehat} may look similar, the two graphs are not the same, which can be verified by comparing the values of $\widetilde{\phi}$ over the interval [2,3], for example.
\qquad\endproof

\section{Conclusion}
In conclusion, this paper extends  the concept of scalability to matrices with Laurent polynomial entries and identifies when the  class of  LP$^2$ matrices, are scalable (cf.~Section~\ref{S:LPmain}). Using these results, we developed a new methodology for constructing tight wavelet filter banks and tight wavelet frames (cf.~Section~\ref{S:wavelet}). We illustrated our construction method for 1-D case by appealing to the Fej\'er-Riesz lemma (cf.~Lemma~\ref{lemma:FR}).

\bibliographystyle{siam} 
\bibliography{IEEEabrv,mybibfile}

\end{document}